\documentclass[10pt,reqno,a4paper,oneside,11pt]{amsart}

\usepackage{amsfonts}
\usepackage{mathrsfs}
\usepackage{mathrsfs}
\usepackage{amsfonts}
\usepackage{amssymb}
\usepackage{amsmath}
\usepackage{amsthm}
\usepackage{graphicx}
\usepackage{color}
\usepackage[numbers,sort&compress]{natbib}

\providecommand{\U}[1]{\protect\rule{.1in}{.1in}}

\newtheorem{theorem}{Theorem}[section]
\newtheorem{corollary}[theorem]{Corollary}
\newtheorem{lemma}[theorem]{Lemma}
\newtheorem{proposition}[theorem]{Proposition}

\newtheorem{definition}{Definition}[section]
\newtheorem{remark}[theorem]{Remark}

\newtheorem{example}[theorem]{Example}

\theoremstyle{definition}
\theoremstyle{remark}
\numberwithin{equation}{section}

\ifx\pdfoutput\relax\let\pdfoutput=\undefined\fi
\newcount\msipdfoutput
\ifx\pdfoutput\undefined\else
\ifcase\pdfoutput\else
\ifx\paperwidth\undefined\else
\ifdim\paperheight=0pt\relax\else\pdfpageheight\paperheight\fi
\ifdim\paperwidth=0pt\relax\else\pdfpagewidth\paperwidth\fi
\fi\fi\fi

\begin{document}
\pagestyle{myheadings}

\begin{center}
{\huge \textbf{A system of dual quaternion matrix equations with its applications}}	
\footnote{This research is supported by the National Natural
Science Foundation of China [grant number 12371023].

\par
{}* Corresponding author.
\par  Email address: wqw@shu.edu.cn (Q.W. Wang); xlm1817@163.com (L.M. Xie). 
}

\bigskip

{ \textbf{Lv-Ming Xie$^{a}$, Qing-Wen Wang$^{a,b,*}$}}

{\small
\vspace{0.25cm}

$a$. Department of Mathematics and Newtouch Center for Mathematics, Shanghai University, Shanghai 200444, People's Republic of China\\
$b$. Collaborative Innovation Center for the Marine Artificial Intelligence, Shanghai University, Shanghai 200444, People's Republic of China

}

\end{center}

\vspace{1cm}

\begin{quotation}
\noindent\textbf{Abstract:}
 We employ the M-P inverses and ranks of quaternion matrices to establish the necessary and sufficient conditions for solving a system of the dual quaternion matrix equations $(AX, XC) = (B, D)$, along with providing an expression for its general solution. Serving as an application, we investigate the solutions to the dual quaternion matrix equations $AX = B$ and $XC=D$, including $\eta$-Hermitian solutions. Lastly, we design a numerical example to validate the main research findings of this paper.

\vspace{3mm}

\end{quotation}

\section{\textbf{Introduction}}

\par\setlength{\parindent}{2em} Hamilton's discovery of quaternions \cite{Hamilton} opened the door to their widespread applications, spanning various domains such as mechanics, quantum physics, signal processing, and color image processing. Subsequent to this, in 1849, James Cockle introduced the concept of split quaternions, which attracted the attention of scholars due to its relevance to solving matrix equations in control theory. Such as Liu et al. and Yuan et al. have conducted work on solving split quaternion matrix equations, as evidenced by references \cite{Liu1,Liu2,Liu3,Yuan}. Owing to the non-commutative nature of quaternions and split quaternions in multiplication, Segre introduced the concept of commutative quaternions. Following that, researchers Xie et al. \cite{Xie}, Ren et al. \cite{Ren}, and Chen et al. \cite{Chen} have explored the solutions for matrix equation systems involving commutative quaternion matrices. In 1873, Clifford \cite{Clifford} introduced the concepts of dual numbers and dual quaternions. Since then, dual quaternions have discovered extensive utility in fields such as robotics, 3D motion modeling, and computer graphics, etc(see \cite{ChengBJ,DaniilidisBJ,KenrightBJ,WangBJ,WangBJ1}). They have become a fundamental component in solving significant engineering challenges, including the formation control of unmanned aerial vehicles and small satellites. This has captured the interest of numerous scholars.
\par\setlength{\parindent}{2em}In 2022, Ling et al. \cite{Ling} conducted a study on the singular values and low-rank approximations of dual quaternion matrices. In \cite{Zhuang}, Zhuang et al. framed the hand-eye calibration model issue as a solving problem of the matrix equation $AX = YB$. Subsequently, Li et al. \cite{Li1} transformed the matrix equation $AX = ZB$ into a dual quaternion equation $\hat{q}_A \hat{q}_X=\hat{q}_Z \hat{q}_B$ by using dual quaternions. In \cite{Chen1}, Chen et al. have  transformed the hand-eye calibration problems $A^{(i)}X=XB^{(i)}$ and $A^{(i)}X = ZB^{(i)}$ into dual quaternion optimization problems  $\min\|{\textbf{a}x-x\textbf{b}}\|_2$ and  $\min\|{\textbf{a}x-z\textbf{b}}\|_2$, respectively. In the process of solving the hand-eye calibration problem, both references \cite{Li1} and \cite{Chen1} have employed singular value decomposition to provide numerical solutions. On the one hand, there has been limited information available regarding the use of matrix M-P inverse and rank as tools to offer exact solutions for dual quaternion matrix equation systems. On the other hand, in the context of the system of classical matrix equations \begin{equation}\label{eq1.1}
	\left\{\begin{array}{c}
		A X=B, \\
		X C=D,
	\end{array}\right.
\end{equation} a multitude of papers have put forth a range of solutions, including Hermitian solutions \cite{Khatri}, the minimum possible rank of solutions \cite{Mitra}, $(R,S)$-conjugate solutions \cite{Chang2010}, reducible solutions \cite{Nie}, $(P,Q)$-(skew)symmetric extremal rank solutions \cite{Zhang2011}, and so on. To enrich the theory and applications of the system of matrix equations \eqref{eq1.1}, we investigate its solutions with respect to dual quaternions in this paper.  
\par\setlength{\parindent}{2em}The structure of this paper unfolds as follows. In Section 2, we revisit the definitions of dual numbers and dual quaternions, provide the definition of $\eta$-Hermitian dual quaternion matrix, and present a crucial lemma and theorem. We devote Section 3 to establish the necessary and sufficient conditions for the solvability of the system of dual quaternion matrix equations \eqref{eq1.1}, and derive an expression for the general solution when the system \eqref{eq1.1} is consistent. Within the scope of the application, we delve into the solutions and $\eta$-Hermitian solutions of the dual quaternion matrix equations $AX=B$ and $XC=D$. In Section 4, we design a numerical example to validate the key results of this paper.  Finally, we summarize the main content of this paper in Section 5.
\par\setlength{\parindent}{2em}Presently, we offer a succinct overview of the notation and properties employed throughout this paper. Let $\mathbb{R},\mathbb{D},\mathbb{H},\mathbb{DQ}$ be the real number field, the complex number field, the dual numbers, quaternions, dual quaternions, respectively. We denote $\mathbb{DQ}^{m \times n}$ ( or $\left.\mathbb{H}^{m \times n}\right)$ as the set of
all $m \times n$ matrices over $\mathbb{DQ}$ $(\text{ or } \mathbb{H})$. For $A\in\mathbb{H}^{m\times n}$, the symbol $r(A)$ represents the rank of $A$, and $A^*$ stands for the conjugate transpose of $A$. We denote the M-P inverse of $A \in \mathbb{H}^{m \times n}$ as $A^\dagger$, and it fulfills the following equations:
$$
AXA=A, \quad X A X=X, \quad(A X)^{*}=A X, \quad(X A)^{*}=X A.
$$
Moreover, we use the notations $L_A$ and $R_A$ to represent the projectors $I - A^\dagger A$ and $I - A A^\dagger$, respectively. It is evident that
$$
L_A=\left(L_A\right)^*=\left(L_A\right)^2=L_A^{\dagger}, R_A=\left(R_A\right)^2=\left(R_A\right)^*=R_A^{\dagger},
\left(L_A\right)^{\eta^*}=R_{A^{\eta^*}},
\left(R_A\right)^{\eta^*}=L_{A^{\eta^*}}.
$$

\section{\textbf{ Preliminary}}

\par\setlength{\parindent}{2em}In this section, we define dual numbers, dual quaternions, and dual quaternion matrices, and describe key operations related to them. We also present an important theorem and lemma that play a fundamental role in deriving the main outcome.

\subsection{Definition of dual numbers and dual quaternions}
\quad\\
\par\setlength{\parindent}{2em}The set of dual numbers is denoted by {\cite{Qi2022}}
\begin{equation}\label{eq2.1}
\mathbb{D}=\{a=a_{0}+a_{1}\epsilon:a_{0},a_{1}\in \mathbb{R} \text{ and }  \epsilon^2=0\},
\end{equation}
where $\epsilon$ is the infinitesimal unit. We refer to $a_0$ as the real part or standard part of $a$, and $a_1$ as the dual part or infinitesimal part of $a$. The infinitesimal unit $\epsilon$ commutes in multiplication with real numbers, complex numbers, and quaternions. Assume that $a=a_{0}+a_{1}\epsilon, b=b_0+b_1\epsilon\in\mathbb{D}, \gamma\in\mathbb{R}$, then we have
$$
\begin{aligned}
a+b&=(a_{0}+b_{0})+(a_{1}+b_{1})\epsilon,\\
ab&=ba=a_0b_0+(a_0b_1+a_1b_0)\epsilon,\\
\gamma a&=\gamma (a_{0}+a_{1}\epsilon)=\gamma a_{0}+\gamma a_{1}\epsilon.
\end{aligned}
$$
\par\setlength{\parindent}{2em}Now, we provide the definition of a dual quaternion. Denotes the collection of dual quaternions as
\begin{equation}\label{eq2.2}
	\mathbb{DQ}=\{c=c_{0}+c_{1}\epsilon:c_{0},c_{1}\in \mathbb{H} \text{ and }  \epsilon^2=0\},
\end{equation}
where $c_{0}, c_{1} $ are the standard part and the infinitesimal part of $c$, respectively. 
\begin{remark}
	Based on the definition of dual quaternions, just as quaternions are non-commutative under multiplication, dual quaternions similarly do not commute under multiplication.
\end{remark}
\begin{definition}\emph{\cite{Qi2022}}\label{de1.1}
	Let $d=d_{0}+d_{1}\epsilon\in\mathbb{DQ}$. Then the conjugate of $d$ is defined as follows:
	$$
	d^*=d_{0}^*+d_{1}^* \epsilon .
	$$
\end{definition}
In a similar manner, we can provide the definition of dual quaternion matrix along with several relevant properties.
\subsection{The definition of dual quaternion matrix}
\quad\\
\par\setlength{\parindent}{2em} A dual quaternion matrix is denoted by $A=A_0+A_1\epsilon\in\mathbb{DQ}^{m\times n}$, where $A_0,A_1\in\mathbb{H}^{m\times n}$. For $B=B_0+B_1\epsilon\in\mathbb{DQ}^{m\times n}$, if $A_0=B_0 \text{ and }A_1=B_1$ are obeyed, then $A=B$. The conjugate transpose of $A$ is designated as $A^*=A_0^*+A_1^*\epsilon$. Should $A^*=A$ and $A$ is a square dual quaternion matrix, it qualifies as a dual quaternion Hermitian matrix, with both its real part and dual part being quaternion Hermitian matrices.
\begin{definition}\label{de2.2}
	If $A=A_0+A_1\epsilon\in\mathbb{DQ}^{n\times n}$, fulfills condition 
	$$
	A=A^{\eta^*}, \ A^{\eta^*}:=-\eta A^*\eta=-\eta A_0^*\eta+(-\eta A_1^*\eta)\epsilon=A_0^{\eta^*}+A_1^{\eta^*}\epsilon, 
	$$
 where $\eta\in\{i,j,k\}$, then the dual quaternion matrix $A$ is termed an $\eta$-Hermitian matrix.
\end{definition}
\begin{proposition}\label{pro1.1}
	Let $A=A_0+A_1\epsilon,B=B_0+B_1\epsilon\in\mathbb{DQ}^{n\times n}$. Then
	\begin{enumerate}
		\item $(A+B)^{\eta^*}=A^{\eta^*}+B^{\eta^*}$;
		\item $(AB)^{\eta^*}=B^{\eta^*}A^{\eta^*}$;
		\item $(A^{\eta^*})^{\eta^*}=A$.
	\end{enumerate}
\end{proposition}
\begin{proof}
	For (1), we have 
	$$
	\begin{aligned}
	 (A+B)^{\eta^*}&=-\eta(A+B)^*\eta=-\eta\left[(A_0+B_0)+(A_1+B_1)\epsilon\right]^*\eta\\
	 &=-\eta[(A_0+B_0)^*+(A_1+B_1)^*\epsilon]\eta\\
	 &=-\eta[A_0^*+B_0^*+A_1^*\epsilon+B_1^*\epsilon]\eta\\
	 &=-\eta(A_0^*+A_1^*\epsilon)\eta-\eta(B_0^*++B_1^*\epsilon)\eta=A^{\eta^*}+B^{\eta^*}.
	\end{aligned}
	$$
In relation to claim (2), we discover that $$
\begin{aligned}
	(AB)^{\eta^*}&=-\eta(AB)^*\eta=-\eta[(A_0B_0)^*+(A_0B_1+A_1B_0)^*\epsilon]\eta\\
	&=-\eta[(B_0^*A_0^*)+(B_1^*A_0^*+B_0^*A_1^*)\epsilon]\eta=-\eta(B^*A^*)\eta\\
	&=-\eta B^*\eta (-\eta) A^*\eta=B^{\eta^*}A^{\eta^*}.
\end{aligned}
$$ 
The answer to (3) is
	$$
	\begin{aligned}
		(A^{\eta^*})^{\eta^*}&=-\eta(A^{\eta^*})^*\eta=-\eta[-\eta A_0^*\eta+(-\eta)A_1^*\eta\epsilon]^*\eta\\
		&=-\eta(A_0^{\eta^*}+A_1^{\eta^*}\epsilon)^*\eta=-\eta(A_0^{\eta^*})^*\eta+(-\eta)(A_1^{\eta^*})^*\eta\epsilon\\
		&=(A_0^{\eta^*})^{\eta^*}+(A_1^{\eta^*})^{\eta^*}\epsilon=A.
	\end{aligned}
	$$
\end{proof}
\subsection{An important lemma and theorem}
\quad\\
\par\setlength{\parindent}{2em}To solve the system of dual quaternion matrix equations \eqref{eq1.1}, we begin by presenting a lemma and theorem related to quaternion matrix equation systems.

\begin{lemma}\emph{\cite{He2015}}\label{lem2.2}
	Let $A_i,B_i$  and $C_i(i=2,3)$ be given over $\mathbb{H}$. Set
	$$
	\begin{array}{l}
		A_{00}=A_3L_{A_2}, B_{00}=R_{B_2}B_3,C_{00}=C_3-A_3A_2^{\dagger}C_2B_2^{\dagger}B_3,D_{00}=R_{A_{00}}A_3,\\
		\Phi=A_2^{\dagger}C_2B_2^{\dagger}+L_{A_2}A_{00}^{\dagger}C_{00}B_3^{\dagger}-L_{A_2}A_{00}^{\dagger}A_3 D_{00}^{\dagger}R_{A_{00}}C_{00}B_3^{\dagger}+D_{00}^{\dagger}R_{A_{00}}C_{00}B_{00}^{\dagger}R_{B_2}.
	\end{array}
	$$
	Then the system of matrix equations 
	\begin{equation}\label{eq2.5}
		\left\{\begin{array}{c}
			A_2YB_2=C_2,\\A_3YB_3=C_3
		\end{array}\right.
	\end{equation}	 
	is solvable if and only if 
	$$
	R_{A_2}C_2 =0, C_2L_{B_2}=0,R_{A_3}C_3 =0,C_3L_{B_3}=0, R_{A_{00}}C_{00}L_{B_{00}}=0.
	$$
	In this case, the general solution can be expressed as
	$$
	Y=\Phi +L_{A_2}L_{A_{00}}W_1+W_2R_{B_{00}}R_{B_2}+L_{A_2}W_3R_{B_3}+L_{A_3}W_4R_{B_2},
	$$
	where $W_i(i=\overline{1,4})$ are arbitrary matrices over $\mathbb{H}$ with appropriate size.
\end{lemma}

\begin{theorem}\label{th2.2}
	Assume that $A, A_1,B, C,C_1,$ and $D$ are given with appropriate sizes over $\mathbb{H}$. Set
	$$
	\begin{array}{l}
		A_2=R_AA_1, B_2=R_C,C_2=R_AB,A_3=L_A,B_3=C_1L_C,C_3=DL_C,\\
		A_{00}=A_3L_{A_2}, B_{00}=R_{B_2}B_3,
		C_{00}=C_3-A_3A_2^{\dagger}C_2B_2^{\dagger}B_3,D_{00}=R_{A_{00}}A_3,\\
		\Phi=A_2^{\dagger}C_2B_2^{\dagger}+L_{A_2}A_{00}^{\dagger}C_{00}B_3^{\dagger}-L_{A_2}A_{00}^{\dagger}A_3 D_{00}^{\dagger}R_{A_{00}}C_{00}B_3^{\dagger}+D_{00}^{\dagger}R_{A_{00}}C_{00}B_{00}^{\dagger}R_{B_2}.
	\end{array}
	$$
	Then the system
	\begin{equation}\label{eq2.6}
		\left\{\begin{array}{l}
			AX+A_1YR_C=B,\\
			XC+L_AYC_1=D
		\end{array}\right.
	\end{equation}
	is consistent if and only if
	\begin{equation}\label{eq2.7}
		AD=BC, R_{A_i}C_i=0,C_iL_{B_i}=0, (i=2,3), \ R_{A_{00}}C_{00}L_{B_{00}}=0.
	\end{equation}
	In this case, the general solution to the system \eqref{eq2.6} can be expressed as
	$$
	\begin{array}{l}
		X=A^{\dagger}(B-A_1YR_C)+L_A(D-L_AYC_1)C^{\dagger}+L_AU_1R_C,\\
		Y=\Phi+L_{A_2}L_{A_{00}}U_2+U_3R_{B_{00}}R_{B_2}+L_{A_2}U_4R_{B_3}+L_{A_3}U_5R_{B_2},
	\end{array}
	$$
	where $U_i(i=\overline{1,5})$ are arbitrary matrices over $\mathbb{H}$ with appropriate sizes.
\end{theorem}
\begin{proof}
	It is evident that the solvability of the system \eqref{eq2.6} is identical to the system 
	\begin{equation}\label{eq2.8}
		\left\{\begin{array}{l}
			AX=B-A_1YR_C,\\
			XC=D-L_AYC_1.
		\end{array}\right.
	\end{equation}
	By Lemma \ref{lem2.2}, it follows that the system of  matrix equations \eqref{eq2.8} is solvable if and only if 
	$$
	R_A(B-A_1YR_C)=0,(D-L_AYC_1)L_C=0, A(D-L_AYC_1)=(B-A_1YR_C)C,
	$$
	i.e., $AD=BC$ and
	\begin{equation}\label{eq2.9}
		\left\{\begin{array}{l}
			R_AA_1YR_C=R_AB, \\
			L_AYC_1L_C=DL_C,
		\end{array}
		\right. \Longleftrightarrow
		\left\{\begin{array}{l}
			A_2YB_2=C_2, \\
			A_3YB_3=C_3.
		\end{array}\right.
	\end{equation}
	Hence, when the system of  quaternion matrix equations \eqref{eq2.8} is solvable, we obtain
	$$
	X=A^{\dagger}(B-A_1YR_C)+L_A(D-L_AYC_1)C^{\dagger}+L_AU_1R_C.
	$$
	Next, it is only necessary to consider the solution of the system of quaternion matrix equations \eqref{eq2.9}. According to Lemma \ref{lem2.2}, the system \eqref{eq2.9} is consistent if and only if the condition 
	$$
	R_{A_i}C_i=0,C_iL_{B_i}=0, (i=2,3), \ R_{A_{00}}C_{00}L_{B_{00}}=0
	$$ 
	holds. In this case, the general solution of the system \eqref{eq2.9} is given by 
	$$
	Y=\Phi+L_{A_2}L_{A_{00}}U_2+U_3R_{B_{00}}R_{B_2}+L_{A_2}U_4R_{B_3}+L_{A_3}U_5R_{B_2},
	$$
	where random matrices over $\mathbb{H}$ of the proper orders are $U i(i=\overline{2,5})$.
\end{proof}

\section{\textbf{The solution of \eqref{eq1.1} and its applications} }
\par\setlength{\parindent}{2em}
Drawing from the aforementioned theorem and lemma, we can now derive the conditions for the solvability of the system of dual quaternion matrix equations \eqref{eq1.1} using the ranks of quaternion matrices, and present an expression for the general solution of the system \eqref{eq1.1}. Additionally, we provide relevant applications based on these results.
\par\setlength{\parindent}{2em}The following lemma is due to Marsaglia and Styan \cite{Mar1974}, which can be easily generalized to $\mathbb{H}$.
\begin{lemma}\label{lem2.3}
	Suppose that $A \in \mathbb{H}^{n \times m}, B \in \mathbb{H}^{n \times l}, C \in \mathbb{H}^{k \times m}, D \in \mathbb{H}^{l_1 \times l}$ and $E \in \mathbb{H}^{k \times l_2}$ are given, then
		 $$
		 r\left[\begin{array}{cc}A & B L_D \\ R_E C & 0\end{array}\right]=r\left[\begin{array}{ccc}A & B & 0 \\ C & 0 & E \\ 0 & D & 0\end{array}\right]-r(E)-r(D).
		 $$
\end{lemma}
\begin{theorem}\label{th3.1}
	Let $A=A_0+A_1\epsilon\in\mathbb{DQ}^{m\times n},B=B_0+B_1\epsilon\in\mathbb{DQ}^{m\times k},C=C_0+C_1\epsilon\in\mathbb{DQ}^{k\times l}$ and $D=D_0+D_1\epsilon\in\mathbb{DQ}^{n\times l}$ be given. Set
	$$
	\begin{array}{l}
	B_{11}=B_1-A_1(A_0^{\dagger}B_0+L_{A_0}D_0C_0^{\dagger}),D_{11}=D_1-(A_0^{\dagger}B_0+L_{A_0}D_0C_0^{\dagger})C_1,A_{11}=A_1L_{A_0},\\
    C_{11}=R_{C_0}C_1,A_2=R_{A_0}A_{11}, B_2=R_{C_0},C_2=R_{A_0}B_{11},A_3=L_{A_0},B_3=C_{11}L_{C_0},\\
	C_3=D_{11}L_{C_0},A_{00}=A_3L_{A_2}, B_{00}=R_{B_2}B_3,
	C_{00}=C_3-A_3A_2^{\dagger}C_2B_2^{\dagger}B_3,D_{00}=R_{A_{00}}A_3,\\
	\Phi=A_2^{\dagger}C_2B_2^{\dagger}+L_{A_2}A_{00}^{\dagger}C_{00}B_3^{\dagger}-L_{A_2}A_{00}^{\dagger}A_3 D_{00}^{\dagger}R_{A_{00}}C_{00}B_3^{\dagger}+D_{00}^{\dagger}R_{A_{00}}C_{00}B_{00}^{\dagger}R_{B_2}.
	\end{array}
	$$
Then the following statements hold the same meaning: 
\begin{enumerate}
	\item The system \eqref{eq1.1} is consistent.
	\item 
	\begin{align}
		&\begin{aligned}
			R_{A_0}B_0=0,D_0L_{C_0}=0,	
			\label{eq3.1}
		\end{aligned}\\
    	&\begin{aligned}
		A_0D_0=B_0C_0, A_0D_1-B_0C_1=B_1C_0-A_1D_0,
		\label{eq3.2}
    	\end{aligned}\\
        &\begin{aligned}
    	R_{A_2}C_2=0,C_2L_{B_2}=0,R_{A_3}C_3=0,C_3L_{B_3}=0,R_{A_{00}}C_{00}L_{B_{00}}=0.
    	\label{eq3.3}
        \end{aligned}
	\end{align}
    \item The equations in \eqref{eq3.2} hold, together with the fulfillment of the rank equality conditions, where
   \vspace{0.8cm}
 	\begin{align}
  	&\begin{aligned}
  		r\left[\begin{matrix}
  			A_0&B_0
  		\end{matrix}\right]=r(A_0),
  	 r\left[\begin{matrix}
  		C_0\\D_0
  	\end{matrix}\right]=r(C_0),
  	\label{eq3.4}
  	\end{aligned}\\
  	&\begin{aligned}
  		r\left[\begin{matrix}
  			A_0 &B_1 &A_1\\
  			0 &B_0 &A_0
  		\end{matrix}\right]=r\left[\begin{matrix}
  		A_0&A_1\\
  		0& A_0
  	\end{matrix}\right],  r\left[\begin{matrix}
  	A_0&A_1D_0-B_1C_0
  \end{matrix}\right]=r
(A_0),
  	\label{eq3.5}
  	\end{aligned}\\
  	&\begin{aligned}
  		r\left[\begin{matrix}
  			C_0\\
  			B_0C_1-A_0D_1
  		\end{matrix}\right]=r(C_0),
  	 r\left[\begin{matrix}
  	 	C_0 & 0\\
  	 	D_1 & D_0\\
  	 	C_1 & C_0
  	 \end{matrix}\right]=r\left[\begin{matrix}
  	 C_0 & 0\\
  	 C_1 & C_0
   \end{matrix}\right],
  	\label{eq3.6}
  	\end{aligned}\\
  &\begin{aligned}
  	r\left[\begin{matrix}
  		B_1C_1-A_1D_1 & A_0 & B_1C_0-A_1D_0\\
  		C_0 & 0 & 0\\
  		B_0C_1-A_0D_1 & 0 & 0
  	\end{matrix}\right]=r(A_0)+r(C_0).
  	\label{eq3.7}
  \end{aligned}
  \end{align}   
\end{enumerate}
In such circumstances, the general solution of the system \eqref{eq1.1} can be expressed as
$X=X_0+X_1\epsilon,$
where
\begin{align}
	&\begin{aligned}
		X_0=A_0^{\dagger}B_0+L_{A_0}D_0C_0^{\dagger}+L_{A_0}UR_{C_0},
		\label{eq3.8}
	\end{aligned}\\
   &\begin{aligned}
   	X_1=A_0^{\dagger}(B_{11}-A_{11}UR_{C_0})+L_{A_0}(D_{11}-L_{A_0}UC_{11})C_0^{\dagger}+L_{A_0}U_1R_{C_0},
   	\label{eq3.9}
   \end{aligned}\\
   &\begin{aligned}
   	U=\Phi+L_{A_2}L_{A_{00}}U_2+U_3R_{B_{00}}R_{B_2}+L_{A_2}U_4R_{B_3}+L_{A_3}U_5R_{B_2},
   	\label{eq3.10}
   \end{aligned}
\end{align}
and  $U_i(i=\overline{1,5})$ are arbitrary matrices with appropriate sizes.
\end{theorem} 
\begin{proof}
	We divide the proof into two parts.
	\par\setlength{\parindent}{2em}\textbf{Part 1.} According to the definitions of dual quaternion matrices multiplication and the equality of dual quaternion matrices, we can derive the system of dual quaternion matrix equations \eqref{eq1.1} are equivalent to the system of quaternion matrix equations
		\begin{equation}\label{eq2.3}
				\left\{\begin{array}{l}
						A_0X_0=B_0,\\X_0C_0=D_0,\\
						A_0X_1+A_1X_0=B_1,\\
						X_0C_1+X_1C_0=D_1.
					\end{array}\right.
			\end{equation}
Hence, solving the system of matrix equations \eqref{eq1.1} over the dual quaternions is effectively reduced to solving the system of quaternion matrix equations \eqref{eq2.3}.
	\par\setlength{\parindent}{2em}\textbf{Part 2.}$(1)\Longleftrightarrow(2)$. Clearly, the system of matrix equations \eqref{eq2.3} can be split into
\begin{equation}\label{eq3.11}
	\left\{\begin{array}{l}
		A_0X_0=B_0,\\
		X_0C_0=D_0,
	\end{array}\right.
\end{equation}
	and 
\begin{equation}\label{eq3.12}
	\left\{\begin{array}{l}
		A_0X_1+A_1X_0=B_1,\\
		X_0C_1+X_1C_0=D1.
	\end{array}\right.
\end{equation}
Therefore, the system of matrix equations \eqref{eq2.3} has a solution if and only if the system of matrix equations \eqref{eq3.11} and \eqref{eq3.12} have a common solution. According to Lemma \ref{lem2.2}, we can conclude that the system \eqref{eq3.11} is consistent if and only if
$$
R_{A_0}B_0=0,D_0L_{C_0}=0,A_0D_0=B_0C_0.
$$
In this case, the general solution of the system of quaternion matrix equations \eqref{eq3.11} is expressed as 
\begin{equation}\label{eq3.13}
	X_0=A_0^{\dagger}B_0+L_{A_0}D_0C_0^{\dagger}+L_{A_0}UR_{C_0},
\end{equation}
where $U$ is an arbitrary matrix over $\mathbb{H}$. 
\par\setlength{\parindent}{2em}Substituting equation \eqref{eq3.13} into the system \eqref{eq3.12}, we can deduce that
\begin{equation}\label{eq3.14}
	\left\{\begin{array}{l}
		A_0X_1+A_1(A_0^{\dagger}B_0+L_{A_0}D_0C_0^{\dagger}+L_{A_0}UR_{C_0})=B_1,\\
		(A_0^{\dagger}B_0+L_{A_0}D_0C_0^{\dagger}+L_{A_0}UR_{C_0})C_1+X_1C_0=D_1,
	\end{array}\right.
\end{equation}
i.e.,
\begin{equation}\label{eq3.15}
	\left\{\begin{array}{l}
		A_0X_1+A_{11}UR_{C_0}=B_{11},\\
	   X_1C_0+L_{A_0}UC_{11}=D_{11}.
	\end{array}\right.
\end{equation}
By Theorem \ref{th2.2}, we have the system of quaternion matrix equations \eqref{eq3.15} is solvable only when
\begin{equation*}
	R_{A_2}C_2=0,C_2L_{B_2}=0,R_{A_3}C_3=0,C_3L_{B_3}=0,R_{A_{00}}C_{00}L_{B_{00}}=0,
\end{equation*}
and
$$
\begin{aligned}
	&A_0D_{11}=B_{11}C_0,\\
	&\Longleftrightarrow A_0[D_1-(A_0^{\dagger}B_0+L_{A_0}D_0C_0^{\dagger})C_1]=[B_1-A_1(A_0^{\dagger}B_0+L_{A_0}D_0C_0^{\dagger})]C_0,\\
   &\Longleftrightarrow A_0D_1-B_0C_1=B_1C_0-A_1D_0.
\end{aligned}
$$
Since $A_0^{\dagger}B_0+L_{A_0}D_0C_0^{\dagger}$ is a particular solution to the system of matrix equations \eqref{eq3.11}, it satisfies 
$A_0(A_0^{\dagger}B_0+L_{A_0}D_0C_0^{\dagger})=B_0$ and $(A_0^{\dagger}B_0+L_{A_0}D_0C_0^{\dagger})C_0=D_0$. In this situation, the general solution of the system \eqref{eq3.15} can be expressed as
$$
\begin{array}{l}
	X_1=A_0^{\dagger}(B_{11}-A_{11}UR_{C_0})+L_{A_0}(D_{11}-L_{A_0}UC_{11})C_0^{\dagger}+L_{A_0}U_1R_{C_0},\\
	U=\Phi+L_{A_2}L_{A_{00}}U_2+U_3R_{B_{00}}R_{B_2}+L_{A_2}U_4R_{B_3}+L_{A_3}U_5R_{B_2},
\end{array}
$$
where $U_i(i=\overline{1,5})$ are arbitrary matrices with appropriate sizes. At this stage, the general solution expression for the system \eqref{eq1.1} is given as $X=X_0+X_1\epsilon$.
\par\setlength{\parindent}{2em}$(2)\Longleftrightarrow(3)$. We only need to demonstrate \eqref{eq3.1} $\Longleftrightarrow$ \eqref{eq3.4} and \eqref{eq3.3} $\Longleftrightarrow$ \eqref{eq3.5}$-$\eqref{eq3.7}, respectively. It's quite evident that $X_0^{'}:=A_0^{\dagger}B_0+L_{A_0}D_0C_0^{\dagger}$ is a particular solution to the system of matrix equations \eqref{eq3.11}, satisfying both $A_0X_0^{'}=B_0$ and $X_0^{'}C_0=D_0$.
\par\setlength{\parindent}{2em}Referring to Lemma \ref{lem2.3}, we find that equations \eqref{eq3.1} is synonymous with equations \eqref{eq3.4}, and obtain
\begin{align*}
	&R_{A_0}B_0=0\Longleftrightarrow r(R_{A_0}B_0)=0\Longleftrightarrow
	r\left[\begin{matrix}
		A_0 & B_0
	\end{matrix}\right]=r(A_0),\\
   &D_0L_{C_0}=0\Longleftrightarrow r(D_0L_{C_0})=0\Longleftrightarrow 
   r\left[\begin{matrix}
   	C_0 \\
   	 D_0
   \end{matrix}\right]=r(C_0).
\end{align*}
Now, our focus shifts to proving \eqref{eq3.3} $\Longleftrightarrow$ \eqref{eq3.5}$-$\eqref{eq3.7}. According to Lemma \ref{lem2.3} and block Gaussian elimination, the following descriptions hold.
\begin{align*}
	&\begin{aligned}
			R_{A_2}C_2=0&\Longleftrightarrow r(R_{A_2}C_2)=0\Longleftrightarrow
		r\left[\begin{matrix}
			A_2 & C_2
		\end{matrix}\right]=r(A_2),\\
	    &\Longleftrightarrow r\left[\begin{matrix}
	    	R_{A_0}A_{11} & R_{A_0}B_{11}
	    \end{matrix}\right]=r(R_{A_0}A_{11}),\\
       &\Longleftrightarrow r\left[\begin{matrix}
    	A_0 & B_{1}-A_1(A_0^{\dagger}B_0+L_{A_0}D_0C_0^{\dagger}) &A_1L_{A_0}
        \end{matrix}\right]=r\left[\begin{matrix}
        A_0 & A_{1}L_{A_0}
       \end{matrix}\right],\\
      &\Longleftrightarrow r\left[\begin{matrix}
    	A_0 & B_{1}-A_1X_0^{'} &A_1\\
    	0 & 0 & A_0
      \end{matrix}\right]=r\left[\begin{matrix}
    	A_0 & A_{1}\\
    	0 & A_0
      \end{matrix}\right],\\
       &\Longleftrightarrow r\left[\begin{matrix}
       	A_0 & B_{1} &A_1\\
       	0 & B_0 & A_0
       \end{matrix}\right]=r\left[\begin{matrix}
       	A_0 & A_{1}\\
       	0 & A_0
       \end{matrix}\right],
	\end{aligned}\\
    &\begin{aligned}
    	C_2L_{B_2}=0&\Longleftrightarrow r(C_2L_{B_2})=0\Longleftrightarrow 
    	r\left[\begin{matrix}
    		B_2 \\
    		C_2
    	\end{matrix}\right]=r(B_2),\\
     &\Longleftrightarrow 
     r\left[\begin{matrix}
     	R_{C_0}\\
     	R_{A_0}B_{11}
     \end{matrix}\right]=r(R_{C_0}),\\
     &\Longleftrightarrow r\left[\begin{matrix}
     B_{1}-A_1X_0^{'} & A_0 &0\\
     	I & 0 & C_0
     \end{matrix}\right]=r\left[\begin{matrix}
     	C_0 & I
     \end{matrix}\right]+r(A_0),\\
     &\Longleftrightarrow r\left[\begin{matrix}
     	A_0 & A_1D_0-B_1C_0
     \end{matrix}\right]=r(A_0).
    \end{aligned}
\end{align*}
Similarly, we can demonstrate that
$$
R_{A_3}C_3=0\Longleftrightarrow r\left[\begin{matrix}
	C_0 \\
	B_0C_1-A_0D_1
\end{matrix}\right]=r(C_0),
$$
and
$$
C_3L_{B_3}=0\Longleftrightarrow r\left[\begin{matrix}
	C_0 & 0\\
	D_1 & D_0\\
	C_1 & C_0
\end{matrix}\right]=r\left[\begin{matrix}
	C_0 & 0\\
	C_1 & C_0
\end{matrix}\right].
$$
Applying Lemma \ref{lem2.3} to $R_{A_{00}}C_{00}L_{B_{00}}=0$, we obtain
\begin{align*}
	r(R_{A_{00}}C_{00}L_{B_{00}})&=0\Longleftrightarrow r\left[\begin{matrix}
		C_{00} & A_{00} \\
		B_{00} & 0
	\end{matrix}\right]=r(A_{00})+r(B_{00}),\\
    &\Longleftrightarrow r\left[\begin{matrix}
    	C_{3}-A_3A_2^{\dagger}C_2B_2^{\dagger}B_3 & A_{3} & 0 \\
    	B_{3} & 0 & B_2\\
    	0 & A_2 &0
    \end{matrix}\right]=r\left[\begin{matrix}
    A_2\\A_3
    \end{matrix}\right]+r\left[\begin{matrix}
    B_2 &  B_3
   \end{matrix}\right],\\
   &\Longleftrightarrow r\left[\begin{matrix}
	C_{3} & A_{3} & 0 \\
	B_{3} & 0 & B_2\\
	0 & A_2 & -A_2A_2^{\dagger}C_2B_2^{\dagger}B_2
   \end{matrix}\right]=r\left[\begin{matrix}
	A_2\\A_3
   \end{matrix}\right]+r\left[\begin{matrix}
	B_2 &  B_3
   \end{matrix}\right],\\
    &\Longleftrightarrow r\left[\begin{matrix}
    	C_{3} & A_{3} & 0 \\
    	B_{3} & 0 & B_2\\
    	0 & A_2 & -C_2
    \end{matrix}\right]=r\left[\begin{matrix}
    	A_2\\A_3
    \end{matrix}\right]+r\left[\begin{matrix}
    	B_2 &  B_3
    \end{matrix}\right],\\
    &\Longleftrightarrow r\left[\begin{matrix}
    	0 & L_{A_{0}} & D_{11}L_{C_0} \\
    	R_{C_{0}} & 0 & C_{11}L_{C_0}\\
    	-R_{A_0}B_{11} & R_{A_0}A_{11} & 0
    \end{matrix}\right]=r\left[\begin{matrix}
    	L_{A_0}\\R_{A_0}A_{11}
    \end{matrix}\right]+r\left[\begin{matrix}
    	R_{C_0} & C_{11} L_{C_0}
    \end{matrix}\right],\\
    &\Longleftrightarrow r\left[\begin{matrix}
    	0 & {A_{0}} & -B_{11}& A_{11} \\
    	{C_{0}} & 0 & 0&0\\
    	D_{11} &0&0 &L_{A_0}\\
    	C_{11} &0&R_{C_0} &0
    \end{matrix}\right]=r\left[\begin{matrix}
    0 &	L_{A_0}\\
    A_0 & A_{11}
    \end{matrix}\right]+r\left[\begin{matrix}
    0&C_0\\
    	R_{C_0} & C_{11}
    \end{matrix}\right],\\
     &\Longleftrightarrow r\left[\begin{matrix}
    	0 & {A_{0}} & -B_{11}& A_{1} &0\\
    	{C_{0}} & 0 & 0&0&0\\
    	D_{11} &0&0 &I&0\\
    	C_{1} &0&I &0&C_0\\
    	0&0&0&A_0&0
    \end{matrix}\right]=r\left[\begin{matrix}
    	0 &	I\\
    	A_0 & A_{1}
    \end{matrix}\right]+r\left[\begin{matrix}
    	0&C_0\\
    	I & C_{1}
    \end{matrix}\right],\\
    &\Longleftrightarrow r\left[\begin{matrix}
    	B_1C_1-A_1D_1 & {A_{0}} & B_{1}C_0-A_1D_0\\
    	{C_{0}} & 0 & 0\\
    	B_0C_1-A_0D_1 &0&0 
    \end{matrix}\right]=r(A_0)+r(C_0),
\end{align*}
i.e. \eqref{eq3.3} $\Longleftrightarrow$ \eqref{eq3.5}$-$\eqref{eq3.7}.
\end{proof}

As several applications of Theorem \ref{th3.1}, we provide the necessary and sufficient conditions for the existence of solutions and $\eta-$Hermitian solutions to the dual quaternion matrix equations $AX=B$ and $XC=D$.

\begin{corollary}\label{cor3.1}
		Let $A=A_0+A_1\epsilon\in\mathbb{DQ}^{m\times n}$ and $B=B_0+B_1\epsilon\in\mathbb{DQ}^{m\times k}$ be given. Set
		$$
		\begin{array}{l}
			B_{11}=B_1-A_1A_0^{\dagger}B_0,\ A_{11}=A_1L_{A_0}, A_2=R_{A_0}A_{11},\ C_2=R_{A_0}B_{11}.
		\end{array}
		$$
		Then the following descriptions are equivalent:
		\begin{enumerate}
			\item The matrix equation $AX=B$ is solvable.
			\item $R_{A_0}B_0=0,\ R_{A_2}C_2=0$.
			\item $r\left[\begin{matrix}
				A_0&B_0
			\end{matrix}\right]=r(A_0)$, $r\left[\begin{matrix}
			A_0 & B_{1} &A_1\\
			0 & B_0 & A_0
		\end{matrix}\right]=r\left[\begin{matrix}
		A_0 & A_{1}\\
		0 & A_0
	\end{matrix}\right]$.
		\end{enumerate}
In this situation, the general solution of the dual quaternion matrix equation $AX=B$ can be expressed as $X=X_0+X_1\epsilon$, where
$$
\begin{array}{l}

		X_0=A_0^{\dagger}B_0+L_{A_0}W,\\
		X_1=A_0^{\dagger}(B_{11}-A_{11}W)+L_{A_0}W_1,\\
		W=A_2^{\dagger}C_2+L_{A_2}W_2,
\end{array}
$$
and $W_i (i = 1,2)$ represent arbitrary matrices over $\mathbb{H}$ with suitable dimensions.
\end{corollary}
\begin{corollary}\label{cor3.2}
	Suppose that $C=C_0+C_1\epsilon\in\mathbb{DQ}^{m\times n}$, and $D=D_0+D_1\epsilon\in\mathbb{DQ}^{k\times n}$. Define
	$$
	\begin{array}{l}
		D_{11}=D_1-D_0C_0^{\dagger}C_1,\ C_{11}=R_{C_0}C_1,\ B_3=C_{11}L_{C_0},\ C_3=D_{11}L_{C_0}.
	\end{array}
	$$
Then the following statements are equivalent:
\begin{enumerate}
	\item The matrix equation $XC=D$ is solvable.
	\item $D_0L_{C_0}=0,\ C_3L_{B_3}=0$.
	\item $r\left[\begin{matrix}
		C_0 \\
		D_0
	\end{matrix}\right]=r(C_0)$, $r\left[\begin{matrix}
	C_0 & 0\\
	D_1 & D_0\\
	C_1 & C_0
\end{matrix}\right]=r\left[\begin{matrix}
C_0 & 0\\
C_1 & C_0
\end{matrix}\right]$.
\end{enumerate}
In this case, the general solution of the dual quaternion matrix equation $XC=D$ can be expressed as $X=X_0+X_1\epsilon$, where
$$
\begin{array}{l}
	X_0=D_0C_0^{\dagger}+UR_{C_0},\\
	X_1=(D_{11}-UC_{11})C_0^{\dagger}+U_1R_{C_0},\\
	U=C_3B_3^{\dagger}+U_2R_{B_3},
\end{array}
$$
and $U_i (i = 1,2)$ are arbitrary matrices over $\mathbb{H}$ with appropriate sizes.
\end{corollary}
\begin{corollary}\label{cor3.3}
	Consider the $\eta$-Hermitian solutions of the dual quaternion matrix equation $AX=B$, and $B^{\eta^*}=B$. Denote
		$$
	\begin{array}{l}
		B_{11}=B_1-A_1(A_0^{\dagger}B_0+L_{A_0}B_0(A_0^{\eta^*})^{\dagger}),D_{11}=B_1-(A_0^{\dagger}B_0+L_{A_0}B_0(A_0^{\eta^*})^{\dagger})A_1^{\eta^*},\\
		A_{11}=A_1L_{A_0},A_2=R_{A_0}A_{11}, B_2=R_{A_0^{\eta^*}},C_2=R_{A_0}B_{11},C_3=D_{11}L_{A_0^{\eta^*}},\\
		A_{00}=B_2^{\eta^*}L_{A_2}, B_{00}=R_{B_2}A_2^{\eta^*},
		C_{00}=C_3-B_2^{\eta^*}A_2^{\dagger}C_2B_2^{\dagger}A_2^{\eta^*},D_{00}=R_{A_{00}}B_2^{\eta^*},\\
		\Phi=A_2^{\dagger}C_2B_2^{\dagger}+L_{A_2}A_{00}^{\dagger}C_{00}(A_2^{\eta^*})^{\dagger}-L_{A_2}A_{00}^{\dagger}B_2^{\eta^*} D_{00}^{\dagger}R_{A_{00}}C_{00}(A_2^{\eta^*})^{\dagger}+D_{00}^{\dagger}R_{A_{00}}C_{00}B_{00}^{\dagger}R_{B_2}.
	\end{array}
	$$
	Then the following statements hold the same meaning: 
	\begin{enumerate}
		\item The matrix equation $AX=B$ is consistent.
		\item 
		\begin{align}
			&\begin{aligned}		
				A_0B_0=B_0A_0^{\eta^*}, A_0B_1-B_0A_1^{\eta^*}=B_1A_0^{\eta^*}-A_1B_0,
				\label{eq3.17}
			\end{aligned}\\
			&\begin{aligned}
			R_{A_0}B_0=0,		R_{A_2}C_2=0,R_{B_2^{\eta^*}}C_3=0,R_{A_{00}}C_{00}L_{B_{00}}=0.
				\label{eq3.18}
			\end{aligned}
		\end{align}
		\item The equations in \eqref{eq3.17} hold, and
		\vspace{0.8cm}
		\begin{align}
			&\begin{aligned}
				r\left[\begin{matrix}
					A_0&B_0
				\end{matrix}\right]=r(A_0),
				\label{eq3.19}
			\end{aligned}\\
			&\begin{aligned}
				r\left[\begin{matrix}
					A_0 &B_1 &A_1\\
					0 &B_0 &A_0
				\end{matrix}\right]=r\left[\begin{matrix}
					A_0&A_1\\
					0& A_0
				\end{matrix}\right],  r\left[\begin{matrix}
					A_0 & A_1B_0-B_1A_0^{\eta^*}
				\end{matrix}\right]=r
				(A_0),
				\label{eq3.20}
			\end{aligned}\\
			&\begin{aligned}
				r\left[\begin{matrix}
					B_1A_1^{\eta^*}-A_1B_1 & A_0 & B_1A_0^{\eta^*}-A_1B_0\\
					A_0^{\eta^*} & 0 & 0\\
					B_0A_1^{\eta^*}-A_0B_1 & 0 & 0
				\end{matrix}\right]=r(A_0)+r(A_0^{\eta^*})=2r(A_0).
				\label{eq3.21}
			\end{aligned}
		\end{align}   
	\end{enumerate}
	In this case, the general solution of the matrix equation $AX=B$ can be expressed as $X=X_0+X_1\epsilon$, where
$$
X_0=\dfrac{\widetilde{X_0}+\widetilde{X_0}^{\eta^*}}{2},
X_1=\dfrac{\widetilde{X_1}+\widetilde{X_1}^{\eta^*}}{2},
$$
and
	\begin{align}
		&\begin{aligned}
			\widetilde{X_0}=A_0^{\dagger}B_0+L_{A_0}B_0(A_0^{\eta^*})^{\dagger}+L_{A_0}UR_{A_0^{\eta^*}},
			\label{eq3.8}
		\end{aligned}\\
		&\begin{aligned}
			\widetilde{X_1}=A_0^{\dagger}(B_{11}-A_{11}UR_{A_0^{\eta^*}})+L_{A_0}(D_{11}-L_{A_0}UA_{11}^{\eta^*})(A_0^{\eta^*})^{\dagger}+L_{A_0}U_1R_{A_0^{\eta^*}},
			\label{eq3.9}
		\end{aligned}\\
		&\begin{aligned}
			U=\Phi+L_{A_2}L_{A_{00}}U_2+U_3R_{B_{00}}R_{B_2}+L_{A_2}U_4R_{A_2^{\eta^*}}+L_{B_2^{\eta^*}}U_5R_{B_2},
			\label{eq3.10}
		\end{aligned}
	\end{align}
  $U_i(i=\overline{1,5})$ are arbitrary matrices with appropriate sizes.
\end{corollary}
\begin{proof}
	By applying the definitions of dual quaternion matrices multiplication and the equality of dual quaternion matrices, we can establish that the dual quaternion matrix equation $AX=B$ are equivalent to the system of quaternion matrix equations
	\begin{equation}\label{eq3.16}
				\left\{\begin{array}{l}
						A_0X_0=B_0,\\
						A_0X_1+A_1X_0=B_1.\\
					\end{array}\right.
	\end{equation}
Now, we only need to provide the $\eta$-Hermitian solutions to the system of quaternion matrix equations \eqref{eq3.16}.
It is evident that the system \eqref{eq3.16} possess $\eta$-Hermitian solutions if and only if the system 
	\begin{equation}\label{eq3.25}
		\left\{\begin{array}{l}
			A_0\widetilde{X_0}=B_0,\\
			\widetilde{X_0}A_0^{\eta^*}=B_0,\\
			A_0\widetilde{X_1}+A_1\widetilde{X_0}=B_1,\\
			\widetilde{X_1}A_0^{\eta^*}+\widetilde{X_0}A_1^{\eta^*}=B_1
		\end{array}\right.
	\end{equation}
	 has solutions. Indeed, if the system \eqref{eq3.16} has $\eta$-Hermitian solutions $X_0$ and $X_1$, it is clear that $X_0$ and $X_1$ serve as solutions to the system \eqref{eq3.25}. Conversely, if the system \eqref{eq3.25} has solutions $\widetilde{X_0}$ and $\widetilde{X_1}$, then the system \eqref{eq3.16} possesses solutions
	 $$
	 X_0=\dfrac{\widetilde{X_0}+\widetilde{X_0}^{\eta^*}}{2},
	 X_1=\dfrac{\widetilde{X_1}+\widetilde{X_1}^{\eta^*}}{2}.
	 $$
\par\setlength{\parindent}{2em}Furthermore, by employing Theorem \ref{th3.1}, it is possible to establish both the necessary and sufficient conditions for the solvability of the system \eqref{eq3.25}, along with an expression for its general solution.
\end{proof}
\begin{remark}
	By applying the same method, we can obtain $\eta$-Hermitian solutions for the dual quaternion matrix equation $XC=D$, and since the structure of the solutions is nearly identical to the $\eta$-Hermitian solutions of $AX=B$, we omit them here.
\end{remark}

\section{Numerical example}
\begin{example}
	 \emph{Given the dual quaternion matrices:}
	 $$
	\begin{aligned}
	 &A=A_0+A_1\epsilon=\left[\begin{array}{cc}
	 {i}  & 0 \\
	 0  & j
	 \end{array}\right]+\left[\begin{array}{cc}
	k & j \\
	 0 & i
	 \end{array}\right]\epsilon,\\
	 &B=B_0+B_1\epsilon=\left[\begin{array}{cc}
	 i &  -1 \\
	 0 & i
	 \end{array}\right]+\left[\begin{array}{cc}
	 k &  -1+i+j \\
	 -1 & 0
     \end{array}\right]\epsilon,\\
	 &C=C_0+C_1\epsilon=\left[\begin{array}{cc}
	 1+i  & 0 \\
	 j  & k
	 \end{array}\right]+\left[\begin{array}{cc}
	 0 & 1 \\
	 j & 0
	 \end{array}\right]\epsilon,\\
    & D=D_0+D_1\epsilon=\left[\begin{array}{cc}
     	1+i+k  & -j \\
     	-i  & -1
     \end{array}\right]+\left[\begin{array}{cc}
     	2k & 1-j \\
     	-i+2j-k & k
     \end{array}\right]\epsilon.\\
	 \end{aligned}
	 $$
	 \emph{Through calculation, it can be determined that}
	 $$
	 	 A_0D_0=B_0C_0
	 	 =\left[\begin{matrix}
	 		-1+i-j & -k\\
	 		k & -j
	 	\end{matrix}\right], 
     	 A_0D_1-B_0C_1=B_1C_0-A_1D_0
     	 =\left[\begin{matrix}
     	 	-j & -k\\
     	 	-2-i & i
     	 \end{matrix}\right],
	 $$
\emph{and}	 
 	\begin{align*}
	&\begin{aligned}
		r\left[\begin{matrix}
			A_0&B_0
		\end{matrix}\right]=r(A_0)=2, \
		r\left[\begin{matrix}
			C_0\\D_0
		\end{matrix}\right]=r(C_0)=2,
	\end{aligned}\\
	&\begin{aligned}
		r\left[\begin{matrix}
			A_0 &B_1 &A_1\\
			0 &B_0 &A_0
		\end{matrix}\right]=r\left[\begin{matrix}
			A_0&A_1\\
			0& A_0
		\end{matrix}\right]= 4, \ r\left[\begin{matrix}
			A_0&A_1D_0-B_1C_0
		\end{matrix}\right]=r
		(A_0)=2,
	\end{aligned}\\
	&\begin{aligned}
		r\left[\begin{matrix}
			C_0\\
			B_0C_1-A_0D_1
		\end{matrix}\right]=r(C_0)=2,\
		r\left[\begin{matrix}
			C_0 & 0\\
			D_1 & D_0\\
			C_1 & C_0
		\end{matrix}\right]=r\left[\begin{matrix}
			C_0 & 0\\
			C_1 & C_0
		\end{matrix}\right]=4,
	\end{aligned}\\
	&\begin{aligned}
		r\left[\begin{matrix}
			B_1C_1-A_1D_1 & A_0 & B_1C_0-A_1D_0\\
			C_0 & 0 & 0\\
			B_0C_1-A_0D_1 & 0 & 0
		\end{matrix}\right]=r(A_0)+r(C_0)=4.
	\end{aligned}
\end{align*} 
\emph{Thus, by Theorem \ref{th3.1}, we conclude that the system of dual quaternion matrix equations \eqref{eq1.1} is solvable, with the general solution expressed as}	 $$
X=X_0+X_1\epsilon=\left[\begin{matrix}
	1 & i\\
	0 & k
\end{matrix}\right]+\left[\begin{matrix}
	0 & i\\
	j & 1
\end{matrix}\right]\epsilon.
	 $$
\end{example}

\section{\textbf{Conclusion}}

In this paper, we have defined $\eta$-Hermitian matrices in the context of dual quaternions and investigated their relevant properties. Subsequently, leveraging matrix Moore-Penrose inverse and rank, we have established both necessary and sufficient conditions for the solvability of the system of dual quaternion matrix equations \eqref{eq1.1}. Additionally, we have presented an expression for the general solution when the system \eqref{eq1.1} is solvable. In an applied context, we have provided the necessary and sufficient conditions, as well as the general expressions for solutions and $\eta$-Hermitian solutions, for the dual quaternion matrix equations $AX=B$ and $XC=D$. Finally, we have validated the primary research outcomes of this paper through a numerical example. Due to the connection between the hand-eye calibration model and the matrix equation $AX=ZB$, as evidenced by reference \cite{Zhuang}, we will consider the solutions to the more general matrix equation $AX-YB=C$ over the dual quaternion algebra.

\end{document}